\documentclass{amsart}

\usepackage{mathrsfs}
\usepackage{amsmath}
\usepackage[all,cmtip,2cell]{xy}
\usepackage{amssymb}
\usepackage{fullpage}
\usepackage{microtype}
\usepackage{hyperref}

\newtheorem{global-theorem}{Theorem}
\newtheorem{theorem}{Theorem}[section]
\newtheorem{lemma}[theorem]{Lemma}
\newtheorem{corollary}[theorem]{Corollary}
\newtheorem{proposition}[theorem]{Proposition}

\theoremstyle{definition}
\newtheorem{definition}[theorem]{Definition}
\newtheorem{example}[theorem]{Example}
\newtheorem{remark}[theorem]{Remark}

\begin{document}

\title{Crossed Simplicial Group Categorical Nerves}
\author{Scott Balchin}
\address{Department of Mathematics\\
University of Leicester\\
University Road, Leicester LE1 7RH, England, UK}
\email{slb85@le.ac.uk}

\maketitle

\begin{abstract}
We extend the notion of the nerve of a category for a small class of crossed simplicial groups, explicitly describing them using generators and relations.  We do this by first considering a generalised bar construction of a group before looking at twisted versions of some of these nerves.  As an application we show how we can use the twisted nerves to give equivariant versions of certain derived stacks.  
\end{abstract}

\section*{Introduction}

Simplicial constructions give us an expansive toolset to use in the theory of many mathematical topics (see \cite{goerss1999simplicial} for an overview of the theory).  There is a natural question -- what meaningful extensions of the simplex category are there?  Examples of such extensions include Connes' cyclic category $\Lambda$  \cite{connes1995noncommutative}, Segal's category $\Gamma$   \cite{lurie_topos} and the category of finite rooted trees $\Omega$    \cite{dendroidal}.  Of interest to us in this article are the categories which have properties similar to the cyclic category.  The cyclic category has the same combinatorics as the simplex category, with the addition of another generator $\tau_n$ which gives a cyclic action on $[n]$:

\begin{itemize}
\item $\tau_n \delta_i = \delta_{i-1} \tau_{i-1}$ for $1 \leq i \leq n$,
\item $\tau_n \delta_0 = \delta_n$,
\item $\tau_n \sigma_i = \sigma_{i-1}\tau_{n+1}$ for $1 \leq i \leq n$,
\item $\tau_n \sigma_0 = \sigma_n \tau^2_{n+1}$,
\item $\tau_n^{n+1} = 1_n$.
\end{itemize}

From these generators one can see that when we take a cyclic set, that is an element of the presheaf category of $\Lambda$, we have a natural action of $C_{n+1}$, the cyclic group of order $n+1$, on $X_n$.  Crossed simplicial groups, as introduced by Loday and Fiedorowicz    \cite{loday} (and independently by Krasauskas under the name of \textit{skew-simplicial sets}    \cite{krasauskas}),  allow us to consider what other groups we can replace the cyclic groups by and still get a category with combinatorial properties like the cyclic category.

Most constructions that can be done in the simplicial setting have an analogue in the cyclic, and therefore also the crossed simplicial setting.  In this paper, we give the explicit construction of a crossed simplicial version of the nerve and bar constructions of a group $G$ for some special examples of crossed simplicial groups which arise through a classification theorem.  We then extend this idea further and define crossed simplicial group nerves of categories.

One place that the classical nerve construction can be used is in the theory of derived algebraic geometry. It allows us to take a functor valued in groupoids (i.e., a 1-stack) and lift it to a functor valued in simplicial sets (i.e., an $\infty$-stack). If we were, for example, to replace the nerve with the cyclic categorical nerve, then we instead get a functor valued in cyclic sets (i.e., a ``cyclic-$\infty$-stack'').  In the last section of this paper, we will pursue this line of thinking, and investigate the equivariant derived stack of local systems.

\section{Crossed Simplicial Group Objects}
Crossed simplicial groups are a generalisation of the simplex category $\Delta$ to allow group actions. They were mainly introduced as tools for use in functor homology    \cite{loday1998cyclic}, but have recently seen other uses, such as in the theory of structured surfaces   \cite{surface}.  We will begin by giving the basic definitions and properties of crossed simplicial groups before looking at some examples.

\begin{definition}
A \textit{crossed simplicial group} is a category $\Delta \mathfrak{G}$ equipped with an embedding $i \colon  \Delta \hookrightarrow \Delta \mathfrak{G}$ such that:
\begin{enumerate}
\item The functor $i$ is bijective on objects.
\item Any morphism $u \colon  i[m] \to i[n]$ in $\Delta \mathfrak{G}$ can be uniquely written as $i(\phi) \circ g$ where $\phi \colon  [m] \to [n]$ is a morphism in $\Delta$ and $g$ is an automorphism of $i[m]$ in $\Delta \mathfrak{G}$.  We call this decomposition the \textit{canonical decomposition}.
\end{enumerate}
We will leave the usage of the functor $i$ implicit, and just refer to objects of $\Delta \mathfrak{G}$ as $[n]$ for $n \geq 0$.  To every crossed simplicial group $\Delta \mathfrak{G}$ we can assign a sequence of groups $\mathfrak{G}_n = \text{Aut}_{\Delta \mathfrak{G}}([n])$.
\end{definition}

\begin{example}
Any simplicial group is an example of a crossed simplicial group, with trivial actions of $\mathfrak{G}_m$ on $\text{Hom}_\Delta([m],[n])$.
\end{example}

\begin{example}
The most well documented example of a crossed simplicial group is Connes' cyclic category, which is used in the theory of non-commutative geometry (see    \cite{connes1,connes2}), and the theory of cyclic homology (see \cite{loday1998cyclic}). Let $\mathfrak{G}_n = C_{n+1}$ with generator $\tau_{n}$ such that $(\tau_{n})^{n+1} = \text{id}_n$.   Then $\Delta \mathfrak{G} = \Lambda$.  To remain consistent with the notation we will be using throughout, we will now denote this category $\Delta \mathfrak{C}$.
\end{example}

\begin{definition}
There is a crossed simplicial group $\Delta \mathfrak{W}$ called the \textit{Weyl crossed simplicial group} where $\mathfrak{G}_{n} = W_{n+1} = C_2 \wr S_{n+1}$, the Weyl group of the $B_n$ root system.  The groups $W_{n+1}$ are sometimes referred to as the \textit{hyperoctahedral groups} (see    \cite{octahedral}).
\end{definition}

\begin{proposition}[{\cite[Theorem 1.7]{surface}}]\label{classification}
Let $\Delta \mathfrak{G}$ be a crossed simplicial group.
\begin{enumerate}
\item There is a canonical functor $\pi \colon  \Delta \mathfrak{G} \to \Delta \mathfrak{W}$.
\item For every $n \geq 0$, there is an induced short exact sequence of groups
$$1 \to \mathfrak{G}_n' \to \mathfrak{G}_n \to \mathfrak{G}_n'' \to 1$$
where $\mathfrak{G}_n'$ is the kernel and $\mathfrak{G}_n''$ is the image of the homomorphism $\pi_n \colon  \mathfrak{G}_n \to W_{n+1}$.
\item The above short exact sequence assembles to a sequence of functors
$$\Delta \mathfrak{G}_n' \to \Delta \mathfrak{G}_n \to \Delta \mathfrak{G}_n''$$
where $\Delta \mathfrak{G}_n'$ is a simplicial group and $\Delta \mathfrak{G}_n'' \subset \Delta \mathfrak{W}$ is a crossed simplicial subgroup of $\Delta \mathfrak{W}$.
\end{enumerate}
\end{proposition}

As a consequence of Proposition \ref{classification}, we see that the classification of crossed simplicial groups reduces to the classification of crossed simplicial subgroups of $\Delta \mathfrak{W}$.  The following corollary gives these subgroups, where we take the opportunity to fix the generators for all groups that we will be interested in.

\begin{corollary}
Any crossed simplicial group $\Delta \mathfrak{G}$ splits as a sequence of functors
$$\Delta \mathfrak{G}_n' \to \Delta \mathfrak{G}_n \to \Delta \mathfrak{G}_n''$$
such that $\Delta \mathfrak{G} ' $ is a simplicial group and $\Delta \mathfrak{G}''$ is one of the following seven crossed simplicial groups:

\begin{itemize}
\item $\Delta$  - The \textit{trivial} crossed simplicial group.
\item $\Delta \mathfrak{C}$ -  The \textit{cyclic} crossed simplicial group.
$$\mathfrak{C}_n =  C_{n+1} = \langle \tau_n \mid \tau_n^{n+1} = 1 \rangle$$
\item $\Delta \mathfrak{S}$ -  The \textit{symmetric} crossed simplicial group.
$$\mathfrak{S}_n =  S_{n+1}  = \langle \sigma_1, \dots , \sigma_n \mid \sigma_i^{2} = 1, \; \sigma_i \sigma_j = \sigma_j \sigma_i \text{ if } j \neq i \pm 1, \; (\sigma_i \sigma_{i+1})^3 = 1 \rangle$$
\item $\Delta \mathfrak{R}$ - The \textit{reflexive} crossed simplicial group.
$$\mathfrak{R}_n = C_2 = \langle \omega \mid \omega^2 = 1 \rangle$$
\item $\Delta \mathfrak{D}$  - The \textit{dihedral} crossed simplicial group.
$$\mathfrak{D}_n = D_{n+1} = \langle \tau_n, \omega \mid \tau_n^{n+1}= \omega^2=  (\tau_n \omega)^2=1 \rangle $$
\item $\Delta \mathfrak{T}$  - The \textit{reflexosymmetric} crossed simplicial group.
$$\mathfrak{T}_n = T_{n+1} = C_2 \ltimes S_{n+1}$$
$$= \langle \omega, \sigma_1, \dots , \sigma_n \mid \sigma_i^2=\omega^2=(\sigma_i \sigma_{i+1})^3 = 1, \omega \sigma_i = \sigma_i \omega , \sigma_i \sigma_j = \sigma_j \sigma_i\rangle$$
\item $\Delta \mathfrak{W}$ - The \textit{Weyl crossed simplicial group}.
$$\mathfrak{W}_n = W_{n+1} =  C_2 \wr S_{n+1} $$
$$ = \langle \sigma_1, \dots , \sigma_n, \kappa \mid \sigma_i^2=\kappa^2=(\sigma_i \sigma_{i+1})^3 = (\sigma_1 \kappa)^4=(\sigma_i \kappa)^2=1 \rangle$$
\end{itemize}
\end{corollary}

\begin{definition}
We will call the above seven crossed simplicial groups the \emph{simple crossed simplicial groups}.  Note that these crossed simplicial groups have the following inclusion structure in their groups $\mathfrak{G}_n$:
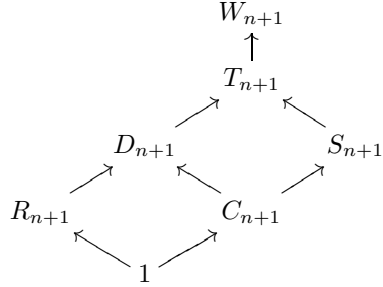
$$\xymatrix@=1em{&& W_{n+1}&\\
&&T_{n+1} \ar[u]&\\
&D_{n+1} \ar[ur]&&S_{n+1} \ar[ul]\\
R_{n+1} \ar[ur]&&C_{n+1} \ar[ur]  \ar[ul]&\\
&1 \ar[ul] \ar[ur]&&}$$
\end{definition}

\begin{example}
An interesting example of a crossed simplicial group arising from the classification theorem uses the braid groups. We denote by $B_n$ the braid group on $n$ braids.  There is a surjection $\mu \colon  B_n \to S_n$ which has kernel $P_n$, the pure braid group.  The family of braid groups $(B_{n+1})_{n \geq 0}$ assembles to a crossed simplicial group $\Delta \mathfrak{B}$ which is given by the extension via the classification theorem:
$$\Delta \mathfrak{P} \to \Delta \mathfrak{B} \to \Delta \mathfrak{S}$$
where $\Delta \mathfrak{P}$ is the simplicial group of pure braids. 
\end{example}

As with the simplex category $\Delta$, our interest with crossed simplicial groups lies in the properties of their presheaf categories.
\begin{definition}
Let $\Delta \mathfrak{G}$ be a crossed simplicial group, $\mathscr{C}$ a category.  A \textit{$\Delta \mathfrak{G}$-object in $\mathscr{C}$} is defined to be a functor:
$$X \colon  (\Delta \mathfrak{G})^{op} \to \mathscr{C}.$$
We shall denote such a functor as $X_\bullet$ with $X_n$ being the image of $[n]$.  If $\lambda \colon  [m] \to [n]$ is a morphism in $\Delta \mathfrak{G}$ we shall write $\lambda^\ast \colon  X_n \to X_m$ for the associated morphism in $X(\lambda)$.  We shall denote the category of all such objects as $\Delta \mathfrak{G} \text{-} \mathscr{C}$.
\end{definition}

For computational reasons it is more convenient to consider a $\Delta \mathfrak{G}$-object as a simplicial object with some extra structure.

\begin{proposition}[{\cite[Lemma 4.2]{loday}}]\label{gobject}
A $\Delta \mathfrak{G}$-object in a category $\mathscr{C}$ is equivalent to a simplicial object $X_\bullet$ in $\mathscr{C}$ with the following additional structure:

\begin{itemize}
\item Left group actions $\mathfrak{G}_n \times X_n \to X_n$.
\item Face relations $d_i(gx) = d_i(g)(d_{g^{-1}(i)}x)$.
\item Degeneracy relations $s_i(gx) = s_i(g)(s_{g^{-1}(i)€}x)$.
\end{itemize}

In particular a $\Delta \mathfrak{G}$-map $f_\bullet \colon  X_\bullet \to Y_\bullet$ is the same thing as a simplicial map such that each of the $f_n \colon  X_n \to Y_n$ is $\mathfrak{G}_n$-equivariant.
\end{proposition}

As a consequence of Proposition \ref{gobject}, we can give concrete combinatorial definitions of crossed simplicial group objects. We will denote the standard face and degeneracy maps of simplicial objects:
$$d_i \colon  [n] \to [n-1],$$
$$s_i \colon  [n] \to [n+1],$$
subject to the usual relations.

\begin{example}\label{come}
For objects over the dihedral category $\Delta \mathfrak{D}$, we have a simplicial objects along with the following additional generators: 
$$\omega_n, \tau_{n}\colon  [n] \to [n]$$
subject to the following relations:
$$\omega_n^2 =\tau_n^{n+1} = \text{id}\colon  [n] \to [n],$$
$$(\tau_n \omega_n)^2 = \text{id}\colon  [n] \to [n],$$
$$d_i \tau_{n} = \tau_{n-1}d_{i-1}\colon  [n] \to [n-1], \; s_i \tau_n = \tau_{n+1} s_{i-1}\colon [n] \to [n+1] \; \text{for } 1 \leq i \leq n,$$
$$d_i \omega_{n} = \omega_{n-1}d_{n-i}\colon  [n] \to [n-1], \; s_i \omega_n = \omega_{n+1} s_{n-i}\colon [n] \to [n+1] \; \text{for } 1 \leq i \leq n,$$
$$d_0 \tau_n = d_n\colon [n] \to [n-1], \; s_0 \tau_{n} = \tau_{n+1}^2 s_n\colon [n] \to [n+1] \; \text{for } n \geq 1.$$
\end{example}

\section{$\Delta \mathfrak{G}$-Bar Constructions of Groups}
Recall that for $G$ a group we can construct the \textit{bar construction} of $G$ which is the simplicial object $\overline{B(G)}$ which in dimension $n$ is equal to $G^{n+1}$, where the face maps act by multiplication and the degeneracy maps act by insertion of the identity element (see    \cite{milnor}).  We can extend this idea to simple crossed simplicial groups, this was done in the cyclic case by Loday \cite[\S 7.3.10]{loday1998cyclic}.

\begin{definition}[Cyclic Bar Construction]
Let $G$ be a group and let $\overline{B(G)}$  be the bar construction of $G$.  We define the \textit{cyclic bar construction} $\overline{B(G)}^{ \mathfrak{C}}$ with the action of the cyclic generator $\tau_n$ on $\overline{B_n(G)}$ being:
$$\tau_n(g_0, \dots , g_n) = (g_n,g_0, \dots , g_{n-1}).$$
This construction is also known in the literature as the \textit{cyclic nerve construction}.
\end{definition}

\begin{lemma}
$\overline{B(G)}^{ \mathfrak{C}}$ is a cyclic group.
\end{lemma}

\begin{proof}
All that needs to be checked is that $\tau_n^{n+1}=\text{id}$.  This follows as the action of $\tau_n$ can be represented as the cycle $(012 \cdots n) \in S_{n+1}$ which has order $n+1$.
\end{proof}

Bar constructions have been considered for all crossed simplicial groups with an application of homology theory.  The definitions that we will give for the bar constructions differ from those which can be found in the literature as the construction presented here is used to generalise the nerve of the category.

\begin{definition}[$\Delta \mathfrak{G}$-Bar Construction]\label{bar}
Let $G$ be a group:
\begin{enumerate}
\item \textit{Symmetric} - We define $\overline{B(G)}^{ \mathfrak{S}}$ to be $\overline{B(G)}$ along with the action of the symmetric generators $\sigma_i$
$$\sigma_i(g_0, \dots ,g_{i-1}, g_i ,\dots , g_n) = (g_0, \dots ,g_{i}, g_{i-1} , \dots , g_n).$$
\item \textit{Reflexive} - We define $\overline{B(G)}^{ \mathfrak{R}}$ to be $\overline{B(G)}$ along with the action of the reflexive generator $\omega$
$$\omega(g_0, \dots , g_n) = (g_n^{-1}, \dots ,  g_0^{-1}).$$
\item \textit{Dihedral} - We define $\overline{B(G)}^{ \mathfrak{D}}$ to be $\overline{B(G)}$ along with the action of the reflexive generator $\omega$ and the operation of the cyclic generator $\tau$ as above.
\item \textit{Reflexosymmetric} - We define $\overline{B(G)}^{ \mathfrak{T}}$ to be $\overline{B(G)}$ along with the action of the reflexive generator $\omega$ and the symmetric generators $\sigma_i$ as above.
\item \textit{Weyl} - We define $\overline{B(G)}^{ \mathfrak{W}}$ to be $\overline{B(G)}$ along with the action of the symmetric generators $\sigma_i$ as above.  Additionally we have the generator $\kappa$ which acts via
$$\kappa(g_0, \dots , g_n) = (g_0^{-1}, g_1, \dots , g_n).$$
\end{enumerate}
\end{definition}

\begin{proposition}\label{barprop}
Let $\Delta \mathfrak{G}$ be a simple crossed simplicial group.  Then for a group $G$, we have the $\Delta \mathfrak{G}$-bar construction $\overline{B(G)}^{\mathfrak{G}}$, as in Definition \ref{bar},  is a $\Delta \mathfrak{G}$-group.
\end{proposition}

\begin{proof}
This can be proved case by case, showing that the generators abide to the combinatorics.  We will show that the generators above satisfy the group axioms.  It can be checked that these generators respect the face and degeneracy map operations.
\begin{enumerate}
\item \textit{Symmetric} - $\langle \sigma_1, \dots , \sigma_n \mid \sigma_i^{2} = 1, \; \sigma_i \sigma_j = \sigma_j \sigma_i \text{ if } j \neq i \pm 1, \; (\sigma_i \sigma_{i+1})^3 = 1 \rangle$.  \\
The first two relations are trivial, so we will only show the last one.
\begin{align*}
(\sigma_i \sigma_{i+1})^3(g_0, \dots , g_{i-1}, g_{i}, g_{i+1}, \dots , g_n) &=(\sigma_i \sigma_{i+1})^2(g_0, \dots , g_{i+1}, g_{i-1}, g_{i}, \dots , g_n)\\
&=(\sigma_i \sigma_{i+1})(g_0, \dots , g_{i}, g_{i+1}, g_{i-1}, \dots , g_n)\\
&=(g_0, \dots , g_{i-1}, g_{i}, g_{i+1}, \dots , g_n)
\end{align*}
\item \textit{Reflexive} - $ \langle \omega \mid \omega^2 = 1 \rangle$.\\ 
This case is obvious as we have $(g_i^{-1})^{-1}=g_i$.
\item \textit{Dihedral} - $\langle \tau_n, \omega \mid \tau_n^{n+1}= \omega^2=  (\tau_n \omega)^2=1 \rangle$. \\ 
We have already shown the validity of the cyclic and reflexive operator, therefore we need only show the final relation:
\begin{align*}
\tau_n \omega \tau_n \omega(g_0, \dots, g_n) &= \tau_n \omega \tau_n(g_n, \dots, g_0)\\
&=\tau_n \omega (g_0, g_n, \dots, g_1)\\
&=\tau_n  (g_1, \dots,g_n, g_0)\\
&= (g_0, \dots, g_n)
\end{align*}
\item \textit{Reflexosymmetric} - $\langle \omega, \sigma_1, \dots , \sigma_n \mid \sigma_i^2=\omega^2=(\sigma_i \sigma_{i+1})^3 = 1, \omega \sigma_i = \sigma_i \omega , \sigma_i \sigma_j = \sigma_j \sigma_i\rangle.$\\
This follows from the symmetric and reflexive case.
\item \textit{Weyl} - $\langle \sigma_1, \dots , \sigma_n, \kappa \mid \sigma_i^2=\kappa^2=(\sigma_i \sigma_{i+1})^3 = (\sigma_1 \kappa)^4=(\sigma_i \kappa)^2=1 \rangle$.   \\
Here the only trivial relation is $(\sigma_1 \kappa)^4 = \text{id}$:
\begin{align*}
(\sigma_1 \kappa)^4 (g_0, g_1, \dots , g_n) &= (\sigma_1 \kappa)^3 (g_1, g_0^{-1}, \dots , g_n)\\
&=(\sigma_1 \kappa)^2 (g_0^{-1}, g_1^{-1}, \dots , g_n)\\
&=(\sigma_1 \kappa) (g_1^{-1}, g_0, \dots , g_n)\\
&=(g_0, g_1, \dots , g_n)
\end{align*}
\end{enumerate}
\end{proof}

\section{$\Delta \mathfrak{G}$-Nerves of Categories}
We can extend the idea of the bar construction further than just groups. In fact we can construct a nerve on a category $\mathscr{C}$ and endow it with a $\Delta \mathfrak{G}$-structure provided that $\mathscr{C}$ has certain properties.  This is formalised in the work of Dykerhoff and Kapranov   \cite{surface} where they give a categorical definition of the $\Delta \mathfrak{G}$-categorical nerves.  In the case of the cyclic and dihedral category, this construction has been explicitly constructed by Connes and Loday respectively. We will extend this construction to the remaining simple crossed simplicial groups, in particular giving constructions for the symmetric and Weyl nerve of $\mathscr{C}$, which give us the relevant generators for the remaining case of the reflexosymmetric category.  The way we will do this is by considering a crossed simplicial group nerve of a category to be an $(n+1)$-tuple of composable morphisms $(a_0, \dots , a_n)$, and then using the generators from the bar constructions in the previous sections. In this case we must take special care that the sources and targets of the morphisms still match up.  Note that this construction differs from the classical nerve construction which is defined to be an $n$-tuple of composable morphisms in dimension $n$, this scenario will be covered by the twisted nerve constructions of Section \ref{sec:twist}.

\begin{definition}[Cyclic Nerve]
Let $\mathscr{C}$ be a category, its \textit{cyclic nerve} $N \mathscr{C}^{\mathfrak{C}}$ is defined to be the simplicial object such that in degree $n$ we have the $(n+1)$ maps in a diagram of the form:
$$\xymatrix{x_0 \ar[r]^{a_0} & x_1 \ar[r]^{a_1} & \cdots \ar[r]^{a_{n-1}}& x_n \ar[r]^{a_n} & x_0}$$
with the cyclic operator $\tau_{n}$ being the cyclic rotation of the diagram:
$$\tau_n\Big(\xymatrix{x_0 \ar[r]^{a_0} & x_1 \ar[r]^{a_1} & \cdots \ar[r]^{a_{n-1}}& x_n \ar[r]^{a_n} & x_0}\Big)$$
$$=\xymatrix{x_n \ar[r]^{a_n} & x_0 \ar[r]^{a_0} & \cdots \ar[r]^{a_{n-2}}& x_{n-1} \ar[r]^{a_{n-1}} & x_n}$$
\end{definition}

This construction works in all generality because the sources and targets of the morphisms always match up.  However, for the reflexive case we will need to be able to reverse the direction of all of the morphisms, so we will require the category to have some further properties.  This property is encoded in the notion of a \textit{dagger category}   \cite{Lambek1999293}.

\begin{definition}
A \textit{dagger category} is a category $\mathscr{C}$ equipped with an involutive functor $\dag\colon \mathscr{C}^{op} \to \mathscr{C}$ that is the identity on objects.  That is, to every morphism $f\colon A \to B$ in $\mathscr{C}$, we associate to it $f^\dag\colon B \to A$ such that for all $f\colon A \to B$ and $g: B \to C$
\begin{itemize}
\item $\text{id}_A = \text{id}_A^\dag\colon A \to A$.
\item $(g \circ f)^\dag = f^\dag \circ g^\dag :C \to A$.
\item $f^{\dag \dag} = f\colon A \to B$.
\end{itemize}
Note, that in particular, a groupoid has a dagger structure, with $f^\dag = f^{-1}$.
\end{definition}

\begin{definition}[Dihedral Nerve]
Let $\mathscr{C}$ be a dagger category, its \textit{dihedral nerve} $N \mathscr{C}^{\mathfrak{D}}$ is defined to be the simplicial object such that in degree $n$ we have the $(n+1)$ maps in a diagram of the form:
$$\xymatrix{x_0 \ar[r]^{a_0} & x_1 \ar[r]^{a_1} & \cdots \ar[r]^{a_{n-1}}& x_n \ar[r]^{a_n} & x_0}$$
with the reflexive operator $\omega$ being begin defined as follows:
$$\omega\Big(\xymatrix{x_0 \ar[r]^{a_0} & x_1 \ar[r]^{a_1} & \cdots \ar[r]^{a_{n-1}}& x_n \ar[r]^{a_n} & x_0}\Big)$$
$$=\xymatrix{x_0 \ar[r]^{a_n^\dag} & x_{n} \ar[r]^{a_{n-1}^\dag} & \cdots \ar[r]^{a_{1}^\dag}& x_1 \ar[r]^{a_0^\dag} & x_0}$$
and the cyclic operators $\tau_n$ as before.
\end{definition}

If we wish to consider the symmetric nerve then we face further limitations.  We would like for the symmetric operator $\sigma_i$ to swap the morphisms $a_i$ and $a_{i-1}$, while also sending $a_i \to a_i^\dag$ and $a_{i-1} \to a_{i-1}^\dag$ to match the fact that $S_2 \cong \mathbb{Z}/2\mathbb{Z}$.  In particular this forces $\text{target}(a_{i-2}) = \text{target}(a_i)$.  This condition therefore requires all of the objects appearing in the diagrams to be identical.

\begin{definition}[Symmetric Nerve]
Let $\mathscr{C}$ be a dagger category, its \textit{symmetric nerve} $N \mathscr{C}^{\mathfrak{S}}$ is defined to be the simplicial object such that in degree $n$ we have the $(n+1)$ maps in a diagram of the form:
$$\xymatrix{ x \ar[r]^{a_0} & x \ar[r]^{a_1} & \cdots \ar[r]^{a_{n-1}}&x \ar[r]^{a_n} & x}$$
The symmetric operator $\sigma_i$ acts on the diagram as follows:
$$\sigma_i\Big(\xymatrix{ x \ar[r]^{a_0} & x \ar[r]^{a_1} & \cdots \ar[r]^{a_{i-2}} & x \ar[r]^{a_{i-1}} & x \ar[r]^{a_i} & x \ar[r]^{a_{i+1}} & \cdots \ar[r]^{a_{n-1}}&x \ar[r]^{a_n} & x}\Big)$$
$$=\xymatrix{ x \ar[r]^{a_0} & x \ar[r]^{a_1} & \cdots \ar[r]^{a_{i-2}} & x \ar[r]^{a_{i}^\dag} & x \ar[r]^{a_{i-1}^\dag} & x \ar[r]^{a_{i+1}} & \cdots \ar[r]^{a_{n-1}}&x \ar[r]^{a_n} & x}$$
\end{definition}

\begin{definition}[Weyl Nerve]
Let $\mathscr{C}$ be a dagger category, its \textit{Weyl nerve} $N \mathscr{C}^{ \mathfrak{W}}$ is defined to be the simplicial object such that in degree $n$ we have the $(n+1)$ maps in a diagram of the form:
$$\xymatrix{ x \ar[r]^{a_0} & x \ar[r]^{a_1} & \cdots \ar[r]^{a_{n-1}}&x \ar[r]^{a_n} & x}$$
with the symmetric operators $\sigma_{i}$ as above, and the operator $\kappa$ acts on the diagram as follows:
$$\kappa\Big(\xymatrix{ x \ar[r]^{a_0} & x \ar[r]^{a_1} & \cdots \ar[r]^{a_{n-1}}&x \ar[r]^{a_n} & x}\Big)$$
$$=\xymatrix{ x \ar[r]^{a_0^\dag} & x \ar[r]^{a_1} & \cdots \ar[r]^{a_{n-1}}&x \ar[r]^{a_n} & x}$$
\end{definition}

\begin{proposition}
The cyclic (resp. dihedral, symmetric, Weyl) nerve is a cyclic (resp. dihedral, symmetric, Weyl) set.
\end{proposition}

\begin{proof}
We can apply the proof of Theorem \ref{barprop}, which has identical generators.  The only extra data that needs to be proved is the matching of sources and targets of maps, which has been taken care of in the construction.
\end{proof}

\section{Twisted Cyclic and Dihedral Categorical Nerves}\label{sec:twist}

We now consider a twisted cyclic and dihedral version of the bar construction.  In this case we will take only $n$ copies of $G$ in dimension $n$, which is the same as in the case of the classical nerve construction of a category $N\mathscr{C}$. We will again begin by considering how to do such a construction on a group $G$, this was done in the cyclic case by Loday \cite[\S 7.3.3]{loday1998cyclic}.

\begin{definition}[Twisted Cyclic Nerve]
Let $G$ be a group and $z \in G$ a central element.  We construct the \textit{$z$-twisted cyclic nerve of $G$, denoted by $B(G,z,\Delta \mathfrak{C})$} by first constructing the nerve $BG$ and defining the action of the cyclic generator $\tau_n$ on $B_n G$ by:
$$\tau_n(g_1, \dots , g_n) = \left(  z(g_1 g_2 \cdots g_n)^{-1}, g_1, \dots , g_{n-1} 	\right)$$
\end{definition}

\begin{lemma}
 $B(G,z,\Delta \mathfrak{C})$ is a cyclic group.  In particular if $z=1$ we get a canonical cyclic structure on $BG$.
\end{lemma}

\begin{proof}
To show that this defines a cyclic structure we must show that $\tau_n^{n+1}=id$.  Observe that
$$\tau_n^{n+1}(g_1, \dots , g_n) = (zg_1z^{-1}, \dots , zg_nz^{-1})$$
This is the identity because we have chosen $z$ to be a central element.
\end{proof}

Below we give a twisted nerve construction which works for the dihedral category.

\begin{definition}[Twisted Dihedral Nerve]
Let $G$ be a group and $z \in G$ a central element of order two.  We construct the \textit{$z$-twisted dihedral nerve of $G$, denoted by $B(G,z,\Delta \mathfrak{D})$} by first constructing the nerve $BG$ and defining the action of the cyclic generator $\tau_n$ on $B_n G$ by: 
$$\tau_n(g_1, \dots , g_n) = \left(  z(g_1 g_2 \cdots g_n)^{-1}, g_1, \dots , g_{n-1} 	\right)$$
and the action of the reflexive generator $\omega$ to be:
$$\omega(g_1, \dots , g_n) = \left(  zg_n^{-1}, \dots , zg_1^{-1} 	\right)$$
\end{definition}

\begin{proposition}\label{dtwist}
For a group $G$, and a central element of order two $z \in G$,  the $z$-twisted dihedral nerve is an example of a dihedral set.
\end{proposition}

\begin{proof}
The $z$-twisted cyclic nerve already gives us a partial proof with the generator $\tau_n$.  Therefore we need only show that the generator $\omega$ follows the group laws.
$$\omega^2(g_1, \dots, g_n) = (z g_1 z^{-1}, \dots , z g_n z^{-1}) =   \text{id by centrality of } z .$$
$$(\tau_n \omega)^2 = (z^2 g_1, g_2, \dots , g_n) = \text{id by the fact } z \text{ has order } 2.$$
\end{proof}

As we did in the previous section, it would be nice to extend this to a categorical nerve setting.  If we do not twist by any elements, it is possible to construct the cyclic and dihedral twisted categorical nerves whenever the category can be endowed with a dagger structure with the additional property:
$$(\ast)\colon f^\dag f \cong  id_A.$$
That is, $f^\dag$ acts up to isomorphism like an inverse.  This can be worded as we require all morphisms in $\mathscr{C}$ to be unitary.  An example of such a category would be any groupoid.

\begin{definition}[Twisted Cyclic Categorical Nerve]
Let $\mathscr{C}$ be a dagger category along with property $(\ast)$.  Its \textit{twisted cyclic nerve} $\widetilde{N \mathscr{C}^{\mathfrak{C}}}$ is defined to be the simplicial object such that in degree $n$ we have the $n$ maps in a diagram of the form:
$$\xymatrix{x_0 \ar[r]^{a_1} & x_1 \ar[r]^{a_1} & \cdots \ar[r]^{a_{n}}& x_n }$$
with the cyclic operator $\tau_{n}$ being defined as follows:
$$\xymatrixcolsep{3pc}\xymatrix{x_n \ar[r]^{(a_1 \cdots a_n)^\dag} & x_0 \ar[r]^{a_1} \ar[r] & \cdots \ar[r]^-{a_{n-1}} & x_{n-1}}$$
\end{definition}

We can define in an analogous way the twisted dihedral nerve, $\widetilde{N \mathscr{C}^{\mathfrak{D}}}$, of a category with the same properties by adding in the reflexive action of $\omega$.

\begin{theorem}
Let $\mathscr{C}$ be a dagger category with property $(\ast)$, then its twisted cyclic (resp. dihedral) nerve is a cyclic (resp. dihedral) set.  
\end{theorem}

\begin{proof}
We need only prove that the cyclic generator has the required properties as the reflexive generator is identical to the non-twisted case.  We begin by noting that:
$$\tau_n^2 (\xymatrix{x_0 \ar[r]^{a_1} & x_1 \ar[r]^{a_1} & \cdots \ar[r]^{a_{n}}& x_n })$$
$$=\xymatrixcolsep{3.5pc}\xymatrix{x_{n-1} \ar[rr]^{\left((a_1 \cdots a_n)^\dag a_1 \cdots a_{n-1}\right)^\dag} & \hspace{0.5cm} & x_n \ar[r]^{(a_1 \cdots a_n)^\dag}  & x_0 \ar[r] & \cdots \ar[r]^{a_{n-2}} & x_{n-2}}$$
We see that 
$$\left((a_1 \cdots a_n)^\dag a_1 \cdots a_{n-1}\right)^\dag = (a_n^\dag \cdots a_1^\dag a_1 \cdots a_{n-1})^\dag$$
and by property $(\ast)$, this becomes $(a_n^\dag)^\dag = a_n$.  For this map to get back to its original position it must be shifted $n-1$ times, for a total of $n-1+2=n+1$ applications of $\tau_n$.  Generalising this ideal to the other maps, it is clear that $\tau_n^{n+1}=\text{id}$ as required.
\end{proof}

\begin{corollary}\label{nataction}
Let $\mathscr{C}$ be a dagger category with property $(\ast)$, then its categorical nerve $N\mathscr{C}$ has a natural cyclic and dihedral structure.  In particular, the nerve of a groupoid has a natural cyclic and dihedral structure.
\end{corollary}

\section{Equivariant Derived Moduli}

We now arrive at the second portion of this paper, which deals with applying the categorical nerves that we have developed to the theory of derived algebraic geometry.  Due to Corollary \ref{nataction}, and the conditions required on the categories for the symmetric (resp. Weyl) nerve, we will only consider the cyclic and dihedral nerves in this section as they will work in full generality.  We will not discuss the full technicalities of (derived)-stacks, but instead direct the interested reader to \cite{MR2778590} for a readable overview, or \cite{MR2137288,MR2394633} for the formal theory.

\begin{definition}
A \emph{stack} is a (lax 2-)functor $\textbf{Aff}^{op}_\tau \to \textbf{Grpd}$ from the opposite (2-)category of affine schemes to the (2-)category of groupoids satisfying descent with respect to the Grothendieck topology $\tau$ (see \cite{MR1771927}).  We will denote the category of stacks as $\text{Stk}(\textbf{Aff})$.
\end{definition}

Stacks were introduced as solutions to certain moduli problems.  However, as soon as you want to classify things up to some weaker notion of equivalence, they are not sufficient.  As early as the work of Grothendieck (see \cite{pursuing}), it was realised that one needs to extend the target category to something ``smooth''.  With the homotopification program, the correct category was found, namely $\textbf{sSet}$.  

\begin{definition}
A \emph{higher stack}  is a functor $\textbf{Aff}^{op}_\tau \to \textbf{sSet}$ satisfying hyperdescent with respect to $\tau$.  The category of higher stacks will be denoted $\textbf{Stk}(\textbf{Aff})$.
\end{definition}

Given a stack $\mathscr{X} \in \text{Stk}(\textbf{Aff})$ it is possible to construct a higher stack $N \mathscr{X} \in \textbf{Stk}(\textbf{Aff})$ by taking the nerve of each groupoid $\mathscr{X}(A)$, $A \in \textbf{Aff}$.  Furthermore, one can ``smooth'' the source category $\textbf{Aff}^{op}$ to $\textbf{dAff}^{op} := \textbf{sComm}$ the (homotopy-)category of derived affine schemes, where $\textbf{sComm}$ is the category of simplicial commutative rings.  

\begin{definition}
A \emph{derived stack} is a functor $\textbf{dAff}^{op}_\tau \to \textbf{sSet}$ satisfying hyperdescent with respect to $\tau$.  The category of derived stacks will be denoted $\textbf{Stk}(\textbf{dAff})$.
\end{definition}

Given a higher stack $\mathscr{X} \in \textbf{Stk}(\textbf{Aff})$, there is an inclusion object $j(\mathscr{X}) \in \textbf{Stk}(\textbf{dAff})$, (after taking a suitable fibrant replacement), induced by the inclusion $j_0\colon \textbf{Aff} \to \textbf{dAff}$.  In particular, combining the above ideas, given any stack $\mathscr{X}$, we can construct a derived stack $j(N \mathscr{X})$.  These ideas can be summed up using the following diagrams:
$$\xymatrixcolsep{6pc}\xymatrixrowsep{4pc}\xymatrix{\textbf{Aff}^{op} \ar[r]^{\text{Stacks}} \ar[dr]|{\text{Higher Stacks}} \ar[d]_{j_0} & \textbf{Grpd} \ar[d]^{N} & \text{Stk}(\textbf{Aff}) \ar[d]_N \ar[dr]^{j \circ N}\\
\textbf{dAff}^{op} \ar[r]_{\text{Derived Stacks}}& \textbf{sSet} & \textbf{Stk}(\textbf{Aff}) \ar[r]_j & \textbf{Stk}(\textbf{dAff})}$$

Our intended application is now immediate, for a simple crossed simplicial group $\Delta \mathfrak{G}$, we wish to construct examples of $\Delta \mathfrak{G}$-derived stacks:

\begin{definition}
A \emph{$\Delta \mathfrak{G}$-derived stack} is a functor $\textbf{dAff}^{op}_\tau \to \Delta \mathfrak{G} \textbf{-Set}$ satisfying equivariant hyperdescent with respect to $\tau$ (which can be made exact using Quillen model structures as done in \cite{balchin2}).  The category of $\Delta \mathfrak{G}$-derived stacks will be denoted $\textbf{Stk}^\mathfrak{G}(\textbf{dAff})$.  A similar definition holds for the notion of a $\Delta \mathfrak{G}$-higher stack.
\end{definition}

A whole range of example of $\Delta \mathfrak{G}$-derived stacks can be obtained by just using the nerve constructions.  Take a stack $\mathscr{X}$, and instead of taking the nerve, take the $\Delta \mathfrak{G}$-nerve (or twisted nerve where appropriate) to get a $\Delta \mathfrak{G}$-higher stack.  We can then take a fibrant replacement for the inclusion into the category of $\Delta \mathfrak{G}$-derived stacks.  We will again denote this inclusion functor $j$.

For the remainder of the paper we will only consider the twisted cyclic (resp. dihedral nerve), the reason being is that it renders the following diagram commutative:
$$\xymatrix{\textbf{Grpd} \ar[r]^N \ar[dr]_{\widetilde{N^\mathfrak{C}}}& \textbf{sSet} \\
& \Delta \mathfrak{C} \textbf{-Set} \ar[u]_{i^\ast}}$$
where $i^\ast$ forgets the cyclic action.  Therefore by using the twisted nerve it will be easier to compare the objects that we get with the non-equivariant case.  Of course it would be of interest to consider what happens in the non-twisted nerves also.  This construction allows us to extend the above diagram to the following (in the case of the cyclic twisted nerve):
$$\xymatrixcolsep{6pc}\xymatrixrowsep{2pc}\xymatrix{\textbf{Aff}^{op} \ar[r]^{\text{Stacks}} \ar[ddr]|{\text{Higher Stacks}} \ar[dd]_{j_0} & \textbf{Grpd} \ar[dd]^{N}  \ar@/^4pc/[dddd]_{\widetilde{N^\mathfrak{C}}}& \\
&& \text{Stk}(\textbf{Aff}) \ar[dd]_{\widetilde{N_\mathfrak{C}}} \ar[ddr]^{j \circ \widetilde{N^\mathfrak{C}}} \\
\textbf{dAff}^{op} \ar[r]_{\text{Derived Stacks}} \ar[ddr]|{\Delta \mathfrak{C} \text{-Derived Stacks}}& \textbf{sSet} & \\
&& \textbf{Stk}^\mathfrak{C}(\textbf{Aff}) \ar[r]_j& \textbf{Stk}^\mathfrak{C}(\textbf{dAff})\\
& \Delta \mathfrak{C}\textbf{-Set} \ar[uu]_{i^\ast}	&}$$

\section{$S^1$-Equivariant Derived Local Systems}

We now use the ideas from the previous section to construct the moduli of equivariant derived local systems on spaces with $S^1$-action.  To do this, we first need to introduce the derived stack of local systems. 

\begin{definition}
Let $G$ be an algebraic group defined over a field $k$.  The \emph{classifying stack} $\mathscr{B}G$ assigns to a scheme $U$ the groupoid whose objects are principal $G$-bundles $\pi\colon \mathcal{E} \to U$, and the morphisms being isomorphisms of principal $G$-bundles.  We will simplify notation and write $\mathscr{B}G$ for $j(N \mathscr{B}G)$, the corresponding derived stack.
\end{definition}

\begin{definition}
Let $\mathscr{B}G$ be the derived classifying stack of an algebraic group and $X$ a topological space.  The \emph{derived stack of $G$-local systems on $X$} is the stack
\begin{align*}
\mathbb{R}\textbf{Loc}(X,G)\colon \textbf{dAff}^{op} &\to \textbf{sSet}\\
U &\mapsto \text{Map}(X, |\mathscr{B}G(U)|)
\end{align*}
That is, $\mathbb{R}\textbf{Loc}(X,G)(U)$ is the simplicial set of continuous maps from the space $X$ to the simplicial set $\mathscr{B}G(U)$. 
\end{definition}

To be able to discuss the cyclic version of this stack, we need the correct analogue of the realisation functor.  We will denote by $\textbf{Top}^{S^1}$ the category of topological spaces with an $S^1$-action.

\begin{proposition}[{\cite[Proposition 2.8]{homotopycc}}]\label{index:cycreal}
There exists a cyclic realisation functor $|-|_{\mathfrak{C}}\colon \Delta \mathfrak{C} \textbf{-Set} \to \textbf{Top}^{S^1}$ such that the following diagram commutes up to a natural isomorphism:
$$\xymatrix{& \textbf{Top}^{S^1}   \ar[d]^u \\
\Delta \mathfrak{C} \textbf{-Set} \ar[ur]^{|-|_{\mathfrak{C}}}  \ar[r]_{|i^\ast -|} & \textbf{Top} }$$
where $u$ is the forgetful functor which forgets the circle action, and $|i^\ast -|$ is the realisation of the underlying simplicial set.
\end{proposition}

\begin{definition}
Let $\mathscr{B}G^\mathfrak{C} := j(\widetilde{N\mathscr{B}G^\mathfrak{C}})$ be the cyclic derived classifying stack of an algebraic group and $X$ a topological space with an action of $S^1$.  The \emph{$S^1$-equivariant derived stack of local systems} is the stack
\begin{align*}
\mathbb{R}\textbf{Loc}^\mathfrak{C}(X,G)\colon \textbf{dAff}^{op} &\to \Delta \mathfrak{C} \textbf{-Set}\\
U &\mapsto \text{Map}^{S^1}(X, |\mathscr{B}G^\mathfrak{C}(U)|_\mathfrak{C})
\end{align*}
That is, $\mathbb{R}\textbf{Loc}^\mathfrak{C}(X,G)(U)$ is the cyclic set of continuous maps in $\textbf{Top}^{S^1}$ from the space $X$ to the space $|\mathscr{B}G^\mathfrak{C}(U)|_\mathfrak{C}$. 
\end{definition}

\begin{remark}
We can adjust the above theory for the twisted dihedral nerve by using the fact that there is a pair of adjoint functors $|-|_\mathfrak{D}\colon \Delta \mathfrak{D} \textbf{-Set} \rightleftarrows \textbf{Top}^{O(2)}\colon S_\mathfrak{D}(-)$ between the categories of dihedral sets and topological spaces with $O(2)$-action.
\end{remark}

The following theorem explains our choice of terminology, the fact the above construction really is doing something equivariant.

\begin{theorem}\label{thrm:equi}
Let $X \in \textbf{Top}^{S^1}$ be a topological space with an action of $S^1$.  Denote by $X/S^1$ the orbit space of $X$, i.e., the space obtained by identifying points of $X$ in the same orbit. Then
$$\mathbb{R}\textbf{Loc}^\mathfrak{C}(X,G) \simeq \mathbb{R}\textbf{Loc}(X/S^1,G).$$
\end{theorem}

\begin{proof}
We can prove this by looking at each element $\text{Map}^{S^1}(X, |\mathscr{B}G^\mathfrak{C}(U)|_\mathfrak{C})$.  First of all we use a result from Loday \cite[\S 7.3.5]{loday1998cyclic} which states that the cyclic realisation of the twisted nerve construction of a group $G$ has trivial $S^1$-action when twisting by the identity element.  As every groupoid is equivalent to the disjoint union of groups, we can conclude that the action of $S^1$ on $|\mathscr{B}G^\mathfrak{C}(U)|_\mathfrak{C}$ is also trivial.  Due to the action being trivial, a general result about $S^1$-spaces, such as in \cite[\S 1.1]{MR1413302}, allows us to move from mapping spaces in $\textbf{Top}^{S^1}$ to $\textbf{Top}$ in the following manner:
$$\text{Map}^{S^1}(X, |\mathscr{B}G^\mathfrak{C}(U)|_\mathfrak{C}) \simeq \text{Map}(X/S^1,|\mathscr{B}G(U)|).$$
The result then follows from this observation.
\end{proof}

\begin{corollary}
If $X \in \textbf{Top}^{S^1}$ has trivial $S^1$-action then
$$\mathbb{R}\textbf{Loc}^\mathfrak{C}(X,G) \simeq \mathbb{R}\textbf{Loc}(X,G).$$
\end{corollary}

\begin{example}
To conclude, we compute the $S^1$-equivariant derived stack on an a non-trivial example.  Consider the $S^1$-space $S^3_\text{Hopf}$ to be the 3-sphere along with the action of the Hopf map (i.e., scalar multiplication).
The orbit space $S^3_\text{Hopf}/S^1$ is homotopic to $S^2$. Therefore by Theorem \ref{thrm:equi} we get:
$$\mathbb{R}\textbf{Loc}^\mathfrak{C}(S^3_\text{Hopf},G) \simeq \mathbb{R}\textbf{Loc}(S^2,G) \simeq [\textbf{Spec} \text{ Sym}_k( \mathfrak{g}^\ast[1])/G]$$
where the final equivalence is computed in the literature, for example, \cite[p. 200]{MR3285853}.
\end{example}

\bibliographystyle{siam}
\bibliography{phdbib}

\begin{thebibliography}{10}

\bibitem{octahedral}
{\sc M.~Baake}, {\em Structure and representations of the hyperoctahedral
  group}, Journal of Mathematical Physics, 25 (1984), pp.~3171--3182.

\bibitem{balchin2}
{\sc S.~Balchin}, {\em Three discrete models of planar {L}ie group equivariant
  presheaves}, arXiv:1608.07238,  (2016).

\bibitem{connes1}
{\sc A.~Connes}, {\em Cohomologie cyclique et foncteur $ext^n$}, Comptes Rendue
  A, Sci, Paris S\'{e}r, 296 (1983), pp.~953--958.

\bibitem{connes1995noncommutative}
\leavevmode\vrule height 2pt depth -1.6pt width 23pt, {\em Noncommutative
  geometry}, Academic Press, Inc., San Diego, CA, 1994.

\bibitem{connes2}
{\sc A.~Connes and C.~Consani}, {\em Cyclic structures and the topos of
  simplicial sets}, J. Pure Appl. Algebra, 219 (2015), pp.~1211--1235.

\bibitem{homotopycc}
{\sc W.~Dwyer, M.~Hopkins, and D.~Kan}, {\em {T}he {H}omotopy {T}heory of
  {C}yclic {S}ets}, Transactions of the {A}merican {M}athematical {S}ociety,
  291 (1985), pp.~281--289.

\bibitem{surface}
{\sc T.~Dyckerhoff and M.~Kapranov}, {\em Crossed simplicial groups and
  structured surfaces}, in Stacks and categories in geometry, topology, and
  algebra, vol.~643 of Contemp. Math., Amer. Math. Soc., Providence, RI, 2015,
  pp.~37--110.

\bibitem{loday}
{\sc Z.~Fiedorowicz and J.-L. Loday}, {\em {C}rossed {S}implicial {G}roups and
  their {A}ssociated {H}omology}, Transactions of the {A}merican {M}athematical
  {S}ociety, 326 (1991), pp.~57--87.

\bibitem{goerss1999simplicial}
{\sc P.~Goerss and J.~Jardine}, {\em Simplicial Homotopy Theory}, Progress in
  mathematics (Boston, Mass.) v. 174, Springer, 1999.

\bibitem{pursuing}
{\sc A.~Grothendieck}, {\em {P}ursuing {S}tacks (\'{A} la poursuite des
  {C}hamps)}, 1983.
\newblock Unpublished manuscript -
  \url{https://thescrivener.github.io/PursuingStacks/ps-online.pdf}.

\bibitem{krasauskas}
{\sc R.~Krasauskas}, {\em {S}kew-{S}implicial {G}roups}, Litovskii
  {M}atematicheskii {S}bornik, 27 (1987), pp.~89--99.

\bibitem{Lambek1999293}
{\sc J.~Lambek}, {\em Diagram chasing in ordered categories with involution},
  Journal of Pure and Applied Algebra, 143 (1999), pp.~293--307.

\bibitem{MR1771927}
{\sc G.~Laumon and L.~Moret-Bailly}, {\em Champs alg\'ebriques}, vol.~39 of
  Ergebnisse der Mathematik und ihrer Grenzgebiete. 3. Folge. A Series of
  Modern Surveys in Mathematics [Results in Mathematics and Related Areas. 3rd
  Series. A Series of Modern Surveys in Mathematics], Springer-Verlag, Berlin,
  2000.

\bibitem{loday1998cyclic}
{\sc J.~Loday}, {\em Cyclic Homology}, Die Grundlehren der mathematischen
  Wissenschaften in Einzeldarstellungen, Springer, 1998.

\bibitem{lurie_topos}
{\sc J.~Lurie}, {\em Higher topos theory}, vol.~170 of Annals of Mathematics
  Studies, Princeton University Press, 2009.

\bibitem{MR1413302}
{\sc J.~P. May}, {\em Equivariant homotopy and cohomology theory}, vol.~91 of
  CBMS Regional Conference Series in Mathematics, Published for the Conference
  Board of the Mathematical Sciences, Washington, DC; by the American
  Mathematical Society, Providence, RI, 1996.
\newblock With contributions by M. Cole, G. Comeza\~na, S. Costenoble, A. D.
  Elmendorf, J. P. C. Greenlees, L. G. Lewis, Jr., R. J. Piacenza, G.
  Triantafillou, and S. Waner.

\bibitem{milnor}
{\sc J.~Milnor}, {\em Construction of universal bundles, {II}}, Annals of
  Mathematics, 63 (1956), pp.~430--436.

\bibitem{dendroidal}
{\sc I.~Moerdijk and I.~Weiss}, {\em Dendroidal sets}, Algebraic and Geometric
  Topology, 7 (2007), pp.~1441--1470.

\bibitem{MR2778590}
{\sc B.~To\"en}, {\em Simplicial presheaves and derived algebraic geometry}, in
  Simplicial methods for operads and algebraic geometry, Adv. Courses Math. CRM
  Barcelona, Birkh\"auser/Springer Basel AG, Basel, 2010, pp.~119--186.

\bibitem{MR3285853}
\leavevmode\vrule height 2pt depth -1.6pt width 23pt, {\em Derived algebraic
  geometry}, EMS Surv. Math. Sci., 1 (2014), pp.~153--240.

\bibitem{MR2137288}
{\sc B.~To{\"e}n and G.~Vezzosi}, {\em Homotopical algebraic geometry. {I}.
  {T}opos theory}, Adv. Math., 193 (2005), pp.~257--372.

\bibitem{MR2394633}
\leavevmode\vrule height 2pt depth -1.6pt width 23pt, {\em Homotopical
  algebraic geometry. {II}. {G}eometric stacks and applications}, Mem. Amer.
  Math. Soc., 193 (2008), pp.~x+224.

\end{thebibliography}

\end{document}